\documentclass[12pt]{amsart}

\usepackage{amsmath,amssymb,latexsym} % ams auxiliaries.
\usepackage{fullpage}
\usepackage{etoolbox}
\usepackage{enumitem}
\usepackage{color}
\usepackage[colorinlistoftodos,bordercolor=orange,backgroundcolor=orange!20,linecolor=orange,textsize=scriptsize]{todonotes}
\usepackage{frenchineq}

\newtheorem{theorem}{Theorem} % If subsection, get 3.1.2, etc.

\newtheorem{proposition}{Proposition}
\newtheorem*{claim}{Claim}
\newtheorem{lemma}{Lemma} 

\theoremstyle{definition}

\newtheorem*{example}{Example}

\newtheorem{remark}{Remark}
\def\D{\mathcal{D}}

\def\R{\mathbb{R}}
\def\Rsf{\mathsf{R}}
\def\Z{\mathbb{Z}}

\def\T{\mathsf{T}}

\def\C{\mathsf{C}}

\def\aa{\boldsymbol{a}}
\def\bb{\boldsymbol{b}}

\def\ee{\boldsymbol{e}}

\def\vv{\boldsymbol{v}}
\def\xx{\boldsymbol{x}}
\def\yy{\boldsymbol{y}}

\def\zero{\boldsymbol{0}}
\def\supp{\operatorname{supp}}

\def\ds{\displaystyle}

\def\sd{\operatorname{sd}}

\def\sd{\operatorname{sd}}
\def\det{\operatorname{det}}

\def\proj{\operatorname{proj}}
\def\supp{\operatorname{supp}}

\begin{document} 

\title{Envy-free cake division without assuming the players prefer nonempty pieces}

\author{Fr\'ed\'eric Meunier}
\address{F. Meunier, Universit\'e Paris Est, CERMICS, 77455 Marne-la-Vall\'ee CEDEX, France}
\email{frederic.meunier@enpc.fr}

\author{Shira Zerbib}
\address{S. Zerbib, Department of Mathematics, University of Michigan, Ann Arbor, MI, USA}
\email{zerbib@umich.edu}

\begin{abstract}
Consider $n$ players having preferences over the connected pieces of a cake, identified with the interval $[0,1]$. A classical theorem, found independently by Stromquist and by Woodall in 1980, ensures that, under mild conditions, it is possible to divide the cake into $n$ connected pieces and assign these pieces to the players in an envy-free manner, i.e, such that no player strictly prefers a piece that has not been assigned to her. One of these conditions, considered as crucial, is that no player is happy with an empty piece. We prove that, even if this condition is not satisfied, it is still possible to get such a division when $n$ is a prime number or is equal to $4$. When $n$ is at most $3$, this has been previously proved by Erel Segal-Halevi, who conjectured that the result holds for any $n$. The main step in our proof is a new combinatorial lemma in topology, close to a conjecture by Segal-Halevi and which is reminiscent of the celebrated Sperner lemma: instead of restricting the labels that can appear on each face of the simplex, the lemma considers labelings that enjoy a certain symmetry on the boundary.
\end{abstract}


\maketitle 

\section{Introduction}

An instance of the envy-free cake-cutting problem consists in the following. We are given $n$ players, numbered from $1$ to $n$, and a cake to be divided among them. Since the cuts are achieved with parallel knives, the cake is identified with the segment [0,1] so that knife cuts are just points of this segment.  A {\em division} of the cake is a partition of $[0,1]$ into finitely many intervals of positive length, which we call {\em nonempty pieces} in this context.
For each player $i$, there is a {\em preference function} $p_i$ which, given a division $\D$ of the cake, returns a subset of $\D\cup\{\varnothing\}$. A nonempty piece $I$ being in $p_i(\D)$ means that player $i$ is happy to get $I$; the empty set $\varnothing$ being in $p_i(\D)$ means that player $i$ is happy to get an {\em empty piece}, i.e., to get nothing (e.g., each nonempty piece is partially burnt, a situation the player prefers to avoid).

An {\em envy-free division} of the cake is a division $\D$ such that there exists a map $\pi\colon[n]\rightarrow\D\cup\{\varnothing\}$ (matching players and pieces) that satisfies the following three conditions: 
\begin{enumerate}[label=(\roman*)]
\item\label{i} $\pi(i)\in p_i(\D)$ for all $i\in[n]$.
\item\label{ii} $\D\subseteq\pi([n])$.
\item\label{iii} $\pi(i)=\pi(i')\Longrightarrow (i=i' \mbox{ or } \pi(i)=\varnothing)$.
\end{enumerate}
Condition~\ref{i} means that each player gets a piece she is happy with, condition~\ref{ii} means that each piece is assigned to a player, and condition~\ref{iii} means that the same nonempty piece cannot be assigned to two distinct players. We present now two assumptions on the preference functions. When $n$ is a prime number or is equal to $4$, they will be enough to ensure the existence of an envy-free division. 

Consider a converging sequence $(\D^k)$ of divisions with a fixed number of nonempty pieces. The preference function $p_i$ satisfies the {\em closed preferences} assumption if the following holds: when there is a converging sequence of pieces $(P^k)$ (empty or not) with $P^k\in  p_i(\D^k)$ for all $k$, then $P^{\infty}\in p_i(\D^{\infty})$, where $\D^{\infty}$ and $P^{\infty}$ are respectively limits of $(\D^k)$ and $(P^k)$.  Here, we have to explain which metrics are used for the convergence, but to ease the readability, we postpone this discussion to the end of the introduction. We already emphasize that the chosen metric makes that whether the endpoints belong or not to the intervals does not matter.

The second assumption, which we call the {\em full division} assumption, ensures that when $|\D|=n$, then $\varnothing\notin p_i(\D)$ for every player $i$. In other words, under this assumption, when the cake has been divided into $n$ intervals of positive length, then no player is happy with an empty piece.

The following theorem is our main result. Without the condition on $n$, it has been conjectured by Erel Segal-Halevi~\cite{segal2017fairly}, who proved it for at most $n=3$ players.

\begin{theorem}\label{main}
Consider an instance of the cake-cutting problem with the number $n$ of players being either a prime number or $4$. If the preference function of every player satisfies the closed preferences assumption and the full division assumption, then there exists an envy-free division.
\end{theorem}

We do not know whether this theorem is still valid if $n$ is neither $4$, nor a prime number.

The classical envy-free cake division theorem, attributed in the recent literature to Walter Stromquist~\cite{stromquist1980cut}, but found independently by Douglas Woodall~\cite{woodall1980dividing}, ensures the same conclusion with the additional assumption that $\varnothing\notin p_i(\D)$ for every player $i$ in any case, even when $|\D|<n$, i.e., a player always strictly prefers a nonempty piece to an empty piece. This additional assumption is usually called the ``hungry players'' assumption and it has always been considered as crucial for the conclusion to hold.  The divisions into at most $n$ pieces can be identified to points in the standard $(n-1)$-dimensional simplex $\Delta^{n-1}$ (these points being the cut positions). The approach by Francis Su~\cite{su1999rental} consists in constructing a triangulation of $\Delta^{n-1}$ and in labeling it. The ``hungry players'' assumption implies that the labeling is a ``Sperner labeling'' and that Sperner's lemma can be applied. 
It was thus very surprising to discover that this assumption is unnecessary when $n=3$ at most, and that it might even be unnecessary for larger $n$.

\begin{example}
For a player $i$, a standard preference function is obtained with $$p_i(\D)=\left\{I\in\D\colon\mu_i(I)=\max_{I'\in\D}\mu_i(I')\right\},$$
where $\mu_i$ is any absolutely continuous measure on $[0,1]$. (This is actually the original type of preference functions considered by Stromquist and by Woodall.) Call it an {\em attraction preference function}. In other words, player $i$ has her own way to weigh the pieces of the cake (encoded by $\mu_i$) and facing a division of the cake, she is happy to get any of the heaviest pieces.

Another simple preference function satisfying the two assumptions required by Theorem~\ref{main}, yet making player $i$ sometimes happy with the empty piece, is obtained as follows:
$$
p_i(\D)=\left\{\begin{array}{l@{\hspace{1cm}}l} \ds{\left\{I\in\D\colon\mu_i(I)=\min_{I'\in\D}\mu_i(I')\right\}} & \mbox{if $|\D|=n$} \bigskip\\ \{\varnothing\} & \mbox{otherwise.}\end{array}\right.
$$
In other words, if the cake is divided into $n$ pieces, she is happy with any of the lightest pieces, and if the cake is divided into less than $n$ pieces, she always strictly prefers to get nothing. It models a situation where, for instance, she does not like at all the cake but she will anyway take a piece when there are $n$ nonempty pieces, in order not to offend the cook. Call such a $p_i$ a {\em rejection preference function}. 

In the two cases where all the players possess attraction preference functions or all the players possess rejection preference functions, an envy-free division exists, without any condition on $n$: in the rejection case, this can be seen by a simple adaptation of the rental harmony's proof by Su~\cite{su1999rental}, and in the attraction case, this is the Stromquist-Woodall theorem.

Theorem~\ref{main} shows that when $n$ is a prime number or $4$, an envy-free division exists, even if preference functions of both types are simultaneously present.
\end{example}

Without the full division assumption, an envy-free division (in the sense of conditions~\ref{i},~\ref{ii}, and~\ref{iii} above)  is not necessarily achievable. Imagine for instance the preference functions $p_i$ are such that, no matter which division $\D$ is chosen, we have $p_i(\D)=\{\varnothing\}$ for all $i$: every player strictly prefers to get nothing instead of a piece of positive length (e.g., the cake has been poisoned). Nevertheless, a division satisfying conditions~\ref{i} and~\ref{iii} always exists in this case, as shown by an easy adaptation of Su's proof, which we leave to the reader.

Our main step toward the proof of Theorem~\ref{main} is given by Sperner-type results, where the usual boundary condition of the Sperner lemma is replaced by a symmetry condition. This symmetry comes from the fact that, in the representation of divisions by points in $\Delta^{n-1}$, divisions into at most $n-1$ pieces admit several representatives. These Sperner-type results, close to a conjecture by Segal-Halevi \cite{segal2017fairly}, are Theorem~\ref{thm:sperner-symmetry-prime} and Proposition~\ref{prop:sperner-symmetry-4}, stated and proved in Section~\ref{sec:sperner}. 

Our paper provides another evidence to the importance of combinatorial topology in studying social-choice problems. Recent literature abounds in more examples, especially in the area of fair division. These include a ``secretive-player'' version of the classical envy-free cake division theorem \cite{frick2017achieving} (see also Remark~\ref{secretive} in Section~\ref{proofs}), discrete versions of this  theorem~\cite{discrete-connected,mirzakhani2015sperner}, and consensus-halving~\cite{SiSu03}. In particular, our paper 
contributes to the current vibrant research activity focused on variations and generalizations of Sperner's lemma. Examples, with many references, are discussed in a survey by De Loera et al.~\cite{DGMM}.

\subsection*{Algorithmic aspects}

%We finish this introduction by emphasizing that 
%Regarding existence of approximate envy-free divisions, 

Our proof is constructive in a purely logical sense but does not directly  admit an algorithm for computing a desired  division. Usually, finding such an algorithm in envy-free division problems is a byproduct of applying a Sperner-like theorem,  which often times provides a path-following method for approximating the required division (see e.g., Su~\cite[Section 5]{su1999rental} and Frick et al.~\cite[Section 5]{frick2017achieving}). Thus, it would be interesting to find an algorithmic version of our proof, especially using a path-following method. Such an algorithmic proof would not only provide a practical way to compute a desired envy-free division in our setting, but would also make the associated computational problem -- ``given polytime preference functions, find an envy-free division'' -- amenable for complexity study.

The relevant complexity class here is the {\em PPAD class}, which is, very roughly speaking, the class of problems that can be solved by a path-following method. The PPAD class admits {\em PPAD-complete} problems, i.e., PPAD problems that are at least as hard as any other PPAD problem. The computational problem associated with the classical envy-free cake division theorem is PPAD-complete, as was shown by Deng et al.~\cite{deng2012algorithmic} (this reference also provides an accessible introduction to the PPAD complexity class). The PPAD-completeness of the classical envy-free cake division problem implies that our problem is PPAD-hard (since it is more general), but not PPAD-complete since its belonging to the PPAD class is not established.  

Almost all the computational problems associated to Sperner's lemma are PPAD-complete; see e.g.,~\cite{chen2009complexity} and~\cite{kiraly2013ppad}, and the references therein. However, finding a fully-labeled simplex in the computational problem associated to our Theorem~\ref{thm:sperner-symmetry-prime} (which is our main Sperner-type result) is {\em not} more general than finding a fully-labeled simplex in the usual Sperner's lemma setting. Thus we cannot directly conclude that the computational problem associated to our Sperner-type result is PPAD-hard.

\subsection*{Convergence of intervals and divisions}
The metric we consider on the measurable subsets of $[0,1]$ is the {\em symmetric difference metric} $\delta$, which is defined for two measurable subsets $A,B$ of $[0,1]$ by $\delta(A,B)=\mu(A\triangle B)$, where $\mu$ is the Lebesgue measure. Note that with the symmetric difference metric, $\varnothing$ is limit of a sequence of intervals whose length goes to $0$. The distance considered between two divisions $\D$ and $\D'$ is then simply the Hausdorff distance induced by $\delta$ between $\D\cup\{\varnothing\}$ and $\D'\cup\{\varnothing\}$.

As already mentioned, divisions are usually identified with points in the standard simplex $\Delta^{n-1}$ and not with partitions into intervals. It is mainly to ease the definition of convergence and to avoid to deal with several representative of the limit that we chose to work with partitions into intervals. Nevertheless, our theorem does imply the Stromquist-Woodall theorem when $n$ is a prime number or $4$, and Segal-Halevi's result for $n=2$. Indeed, if the closed preferences assumption is satisfied for the definition with points in the simplex, it is also satisfied for our definition.

\begin{remark}
Another option for defining the divisions and the topology would have been to consider $$D^{n-1}=\big\{(z_1,\ldots,z_{n-1})\in\R^{n-1}\colon 0\leq z_1\leq\cdots\leq z_{n-1}\leq 1\big\}/\sim,$$ where $\sim$ is the equivalence relation defined on $\R^{n-1}$ by $$(z_1,\ldots,z_{n-1})\sim(z_1',\ldots,z_{n-1}')\quad\mbox{if and only if}\quad \{z_1,\ldots,z_{n-1}\}=\{z_1',\ldots,z_{n-1}'\}.$$ Two points in $\R^{n-1}$ are equivalent if the sets consisting of their coordinates are equal. The space $D^{2}$ is the classical dunce hat space. The spaces $D^{n-1}$ have been more systematically studied by Andersen, Marjanovi\'c, and Schori~\cite{andersen1993symmetric}, and generalized by Kozlov~\cite{koz}.  %When $n$ is odd, they are called ``higher-dimensional Dunce hat''.
There is an one-to-one correspondence between the points in $D^{n-1}$ and the divisions of the cake in at most $n$ pieces: the coordinates of a point in $D^{n-1}$ are the endpoints of the intervals (except $0$ and $1$). The Hausdorff metric on $\R^{n-1}$ is compatible with the equivalence relation $\sim$ and induces then a metric on $D^{n-1}$. Convergences of divisions could have been considered according to this metric: it does not change the topology. But we have not made this choice because we have thought that it makes the description of the problem less intuitive and more cumbersome.
\end{remark}

%Consider a converging sequence $(\D_k)$ of divisions with limit $\D_{\infty}$. Let $I$ be an interval in $\D_{\infty}$. Up to taking a subsequence, it is possible to select in each $\D_k$ an interval $I_k$ such that $(I_k)$ converges to $I$.

\section{``Sperner lemmas'' with symmetries}\label{sec:sperner}

The purpose of this section is to state and prove combinatorial fixed point results (Theorem~\ref{thm:sperner-symmetry-prime} and Proposition~\ref{prop:sperner-symmetry-4}) in the spirit of the Sperner lemma. One of these results, which will play a crucial role in our proof of Theorem~\ref{main}, is close to Conjecture 4.15 in the paper by Segal-Halevi \cite{segal2017fairly}, who realized that it would be the main step to a proof of his conjecture on cake-division.  It deals with triangulations of $\Delta^{n-1}$ -- the standard $(n-1)$-dimensional simplex -- that satisfy some symmetry.

\subsection{Statements}
For each $j\in[n]$, we introduce two maps. The first map is the permutation $\rho^j$ on $[n]$ defined by 
$$\rho^j(i)=\left\{\begin{array}{ll} j & \mbox{if $i=1$} \\ i-1 & \mbox{if $2\leq i\leq j$} \\ i & \mbox{if $i\geq j+1$.}\end{array}\right.$$ The second map, denoted $r^j$, is the linear self-map of $\R^n$ defined by $r^j(\ee_i)=\ee_{\rho^j(i)}$ for $i=1,\ldots,n$, where the $\ee_i$ are the unit vectors of the standard basis of $\R^n$. Note that $r^j$ is a bijection that induces a permutation of the facets of $\Delta^{n-1}$, and that when $j=1$, both maps are the identity map.

We identify $\Delta^{n-1}$ with $\{(x_1,\ldots,x_n)\in\R^n_+\colon\sum_{i=1}^nx_i=1\}$. The facet of $\Delta^{n-1}$ defined by $x_j=0$ is the {\em $\widehat{j}$-facet}. Note that the vertices of $\Delta^{n-1}$ are the $\ee_i$'s. For a point $\xx\in\Delta^{n-1}$, we define $J_{\xx}$ to be the set of all $j$ such that $\xx$ belongs to the $\widehat{j}$-facet.

Triangulations and labelings considered in the paper will in general enjoy certain symmetries. We say that a triangulation $\T$ of $\Delta^{n-1}$ is {\em nice} if $r^j(\sigma)\in\T$ for every $j\in[n]$ and every simplex $\sigma\in\T$ included in the $\widehat{1}$-facet. (Such triangulations are called ``friendly'' in Segal-Halevi's paper.) Similarly, consider a labeling $\Lambda$ of the vertices $v$ of $\T$ with subsets of $[n]$. We say such a labeling is {\em nice} if $\Lambda\big(r^j(v)\big)=\rho^j\big(\Lambda(v)\big)$ for every $j\in [n]$ and every vertex $v$ of $\T$ included in the $\widehat{1}$-facet.

%{\blue Shira: I think that Propositions 1 and 2 should be theorems.}

\begin{theorem}\label{thm:sperner-symmetry-prime}
Let $\T$ be a nice triangulation of $\Delta^{n-1}$ and let $\Lambda$ be a nice labeling of its vertices with nonempty proper subsets of $[n]$. If $n$ is a prime number, then there is an $(n-1)$-dimensional simplex $\tau\in\T$ such that it is possible to pick a distinct label in each $\Lambda(u)$ when $u$ runs over the vertices of $\tau$.
\end{theorem}

 We do not know whether Theorem~\ref{thm:sperner-symmetry-prime} is true for any nonprime $n$. Up to additional conditions on the labeling, we are however able to prove that a special case holds when $n=4$.

The {\em supporting face} $\supp(\xx)$ of a point $\xx$ in $\Delta^{n-1}$ is the inclusion-minimal face of $\Delta^{n-1}$ containing this point. Two faces are {\em comparable by inclusion} if one of them is included in the other (they can be equal).

\begin{proposition}\label{prop:sperner-symmetry-4}
Let $\T$ be a nice triangulation of $\Delta^{3}$ such that the supporting faces of any two adjacent vertices are comparable by inclusion. Suppose that $\Lambda$ is a nice labeling of its vertices with nonempty subsets of $\{1,2,3,4\}$ such that for every vertex $v$, the subset $\Lambda(v)$ is either $J_v$ or a singleton not belonging to $J_v$. Then there is a $3$-dimensional simplex $\tau\in\T$ such that it is possible to pick a distinct label in each $\Lambda(u)$ when $u$ runs over the vertices of $\tau$.
\end{proposition}

Replacing $4$ by a prime number $n$ in the statement of Proposition~\ref{prop:sperner-symmetry-4} leads to a valid result, since such a result is a direct consequence of Theorem~\ref{thm:sperner-symmetry-prime}: this latter theorem requires much weaker constraints on the triangulation and the labeling. For the proof of Theorem~\ref{main}, such a result would actually be enough, but we think that Theorem~\ref{thm:sperner-symmetry-prime} has its own merit.

%We do not know whether Theorem~\ref{thm:sperner-symmetry-prime} is true for $n=4$, or for any nonprime $n$.

%Proposition~\ref{prop:sperner-symmetry-prime} implies that Proposition~\ref{prop:sperner-symmetry-4} where $4$ is replaced by a prime number $n$ also holds, since this latter proposition requires much stronger constraints on the triangulation and the labeling. For the proof of Theorem~\ref{main}, such a result would actually be enough, but we think that Proposition~\ref{prop:sperner-symmetry-prime} has its own merit.

\subsection{Preliminaries}

For the proofs, we assume basic knowledge in algebraic topology; see the book by Munkres~\cite{Mun}, and especially Chapter 1 for the notions used hereafter (abstract simplicial complexes, chains, chain maps, etc.). To that traditional tools, we add the following ones. Let $\Rsf$ be the abstract simplicial complex whose vertices are the points of $\R^n$ and whose maximal simplices are all possible $n$-subsets of $\R^n$. (Note that $\Rsf$ is an abstract simplicial complex, two $(n-1)$-dimensional simplices may have intersecting convex hulls, and the vertices of a simplex can be aligned.) For an oriented $(n-1)$-dimensional simplex $[\vv_1,\ldots,\vv_n]$ of $\Rsf$, we define $\det_{\sharp}([\vv_1,\ldots,\vv_n])$ to be $\det(\vv_1,\ldots,\vv_n)$. We extend then the definition of $\det_{\sharp}$ by linearity for all elements in $C_{n-1}(\Rsf,\Z)$.

The following generalization of the Sperner lemma will play a key role in the proof of Theorem~\ref{thm:sperner-symmetry-prime} and Proposition~\ref{prop:sperner-symmetry-4}. Contrary to the latter results, there is no particular assumption on the triangulation $\T$ and on $n$.

\begin{lemma}\label{lem:sperner_det}
Let $\T$ be a coherently oriented triangulation of $\Delta^{n-1}$. Consider a labeling $\lambda$ of the vertices in $V(\T)$ with points in $\Delta^{n-1}$, such that, for every vertex $v\in V(\T)$, the point $\lambda(v)$ lies in the affine hull of $\supp(v)$.
Then $$\left|\sum_{[v_1,\ldots,v_n]}\det(\lambda(v_1),\ldots,\lambda(v_n))\right|=1,$$ where $[v_1,\ldots,v_n]$ runs over the positively oriented $(n-1)$-dimensional simplices of $\T$.
\end{lemma}

\begin{proof}
We proceed by induction on $n$. If $n=1$, the result is immediate: $\lambda(\ee_1)=\ee_1$.

Suppose $n\geq 2$ and let $t$ be the element of $C_{n-1}(\T,\Z)$ that is the formal sum of all positively oriented $(n-1)$-dimensional simplices of $\T$. Our goal is to prove that $(\det_{\sharp}\circ\lambda_{\sharp})(t)=1$, where we interpret $\lambda$ as a simplicial map from $\T$ to $\Rsf$.

Let $\proj\colon\R^n\rightarrow\R^{n-1}$ be the projection on the first $n-1$ coordinates, defined by
$$\proj(\ee_i)=\left\{\begin{array}{ll}\ee'_i& \mbox{if $i\neq n$} \\ \zero & \mbox{otherwise,}\end{array}\right.$$ where the $\ee_i'$ are the unit vectors of the standard basis of $\R^{n-1}$.
We interpret $\proj$ as a simplicial map from $\Rsf$ to $\Rsf'$ too, where $\Rsf'$ is the abstract simplicial complex whose vertices are the points of $\R^{n-1}$ and whose maximal simplices are the $(n-1)$-subsets of $\R^{n-1}$.

\begin{claim}
The following equality holds: $(\det_{\sharp}\circ\lambda_{\sharp})=(-1)^{n-1}(\det_{\sharp}\circ \proj_{\sharp}\circ\partial\circ \lambda_{\sharp})$.
\end{claim}

\begin{proof}
We prove the equality for an oriented simplex $[v_1,\ldots,v_n]$, which is enough to get the conclusion. Define $\aa_i=(a_{1,i},\ldots,a_{n,i})$ to be $\lambda(v_i)$ and $\aa_i'$ to be $\proj(\aa_i)$. By definition of $\proj$, we have  $\aa'_i=(a_{1,i},\ldots,a_{n-1,i})$. If two of the $\aa_i$'s are equal, the left-hand and right-hand terms of the equality to prove are both equal to zero when applied to the considered oriented simplex. We can thus assume that all $\aa_i$'s are distinct.

Compute the left-hand term on the considered oriented simplex: $$\begin{array}{rcl}
(\det_{\sharp}\circ\lambda_{\sharp})([v_1,\ldots,v_n]) & = & \ds{\det\left(\aa_1,\ldots,\aa_n\right)}\smallskip \\ 
& = & \ds{\left|\begin{array}{ccc} a_{1,1} & \cdots & a_{1,n} \\ \vdots & & \vdots \\ a_{n-1,1} & \cdots & a_{n-1,n} \\ 1 & \cdots & 1 \end{array}\right|}\smallskip\\
&  = & \ds{\sum_{i=1}^n(-1)^{n-i}\det\left(\aa'_1,\ldots,\widehat{\aa'_i},\ldots,\aa'_n\right)},
\end{array}
$$ where the second equality follows from row operations and the fact that $\sum_{j=1}^n a_{j,i} = 1$.

Similarly, compute the right-hand term on the considered oriented simplex: 
$$\begin{array}{rcl}
(\det_{\sharp}\circ \proj_{\sharp}\circ\partial\circ\lambda_{\sharp})([v_1,\ldots,v_n]) & = & 
\ds{(\det_{\sharp}\circ \proj_{\sharp})\left(\sum_{i=1}^n(-1)^{i-1}[\aa_1,\ldots,\widehat{\aa_i},\ldots,\aa_n]\right)} \\
& = & \ds{\sum_{i=1}^n(-1)^{i-1}\det\left(\proj(\aa_1),\ldots,\widehat{\proj(\aa_i)},\ldots,\proj(\aa_n)\right)}.
\end{array}
$$
We have thus in any case $(\det_{\sharp}\circ\lambda_{\sharp})([v_1,\ldots,v_n])=(-1)^{n-1}(\det_{\sharp}\circ \proj_{\sharp}\circ\partial\circ\lambda_{\sharp})([v_1,\ldots,v_n])$.
\end{proof}

According to this claim and the commutation of $\partial$, we have
$$(\det_{\sharp}\circ\lambda_{\sharp})(t)=(-1)^{n-1}(\det_{\sharp}\circ \proj_{\sharp}\circ \lambda_{\sharp})\left(\sum_{j=1}^nt_j\right),$$
where $t_j$ is the term in $\partial t$ supported by the $\widehat{j}$-facet of $\Delta^{n-1}$. For $j\neq n$, the $j$th coordinate of $(\proj\circ\lambda)(u)$ for any vertex $u\in V(\T)$ on the $\widehat{j}$-facet is equal to $0$. Hence, $(\det_{\sharp}\circ \proj_{\sharp}\circ \lambda_{\sharp})(t_j)=0$ when $j\neq n$. Therefore,  $(\det_{\sharp}\circ\lambda_{\sharp})(t)=(-1)^{n-1}(\det_{\sharp}\circ \proj_{\sharp}\circ \lambda_{\sharp})(t_n)$.

Let $\T'$ be the triangulation of the $\widehat{n}$-facet of $\Delta^{n-1}$ induced by $\T$. We identify this facet with $\Delta^{n-2}=\{(x_1,\ldots,x_{n-1})\in\R_+^{n-1}\colon\sum_{i=1}^{n-1}x_i=1\}$. The map $\proj\circ\lambda$ is a labeling of the vertices of $\T'$ with elements of $\Delta^{n-2}=\{(x_1,\ldots,x_{n-1})\in\R_+^{n-1}\colon\sum_{i=1}^{n-1}x_i=1\}$, which satisfies the condition of the lemma for $n-1$. The chain $t_n$ is the formal sum of all positively oriented simplices of $\T'$. By induction, we have thus  $\left|(\det_{\sharp}\circ\proj_{\sharp}\circ \lambda_{\sharp})(t_n)\right|=1$, which implies %(and not always $1$ since the identification of the $\widehat{n}$-facet of $\Delta^{n-1}$ with $\Delta^{n-2}$ reverses the orientation when $n$ is even). 
$\left|(\det_{\sharp}\circ \lambda_{\sharp})(t)\right|=1$.
\end{proof}

A lemma by Yakar Kannai~\cite[Lemma 3]{kannai2013using} has the same condition and almost the same conclusion as Lemma~\ref{lem:sperner_det}. Regarding the conclusion, the lemma of Kannai ensures the existence of a subdivision of $\T$ for which the formula of Lemma~\ref{lem:sperner_det} holds, while we know henceforth that it holds for $\T$ itself.

Florian Frick (personal communication) noted that the approach by Andrew McLennan and Rabee Tourky~\cite{mclennan2008using} for proving the Sperner lemma can also be used here to provide an alternative proof of Lemma~\ref{lem:sperner_det}.

%\begin{lemma}\label{lem:r_jdet}
%For any collection of  $n$ vectors $\aa_1,\ldots,\aa_n$ in $\R^n$, we have $$\det[r^j(\aa_1),\ldots,r^j(\aa_n)]=(-1)^{j-1}\det[\aa_1,\ldots,\aa_n].$$
%\end{lemma}
%
%\begin{proof}
%Immediate.
%\end{proof}

For a nonempty subset $S$ of $[n]$, we define the points $\bb^S=(b_1^S,\ldots,b_n^S)$ of $\Delta^{n-1}$ by $$b_i^S=\left\{\begin{array}{ll}\ds{ \frac 1 {|S|}}& \mbox{if $i\in S$}\medskip \\ 0 & \mbox{otherwise.}\end{array}\right.$$
The point $\bb^S$ is the barycenter of the face whose vertex set is $\{ \ee_i\colon i\in S\}$. These points are subject to two easy lemmas, which will be useful in the sequel.

\begin{lemma}\label{lem:cover}
Suppose that $S_1,\ldots,S_n$ are nonempty subsets of $[n]$ such that $\det(\bb^{S_1},\ldots,\bb^{S_n})\neq 0$. Then it is possible to pick in each $S_j$ a distinct element of $[n]$.
\end{lemma}

\begin{proof}
Since the determinant is nonzero, there is a permutation $\pi$ of $[n]$ such that $b_{\pi(i)}^{S_i}$ is nonzero for all $i\in[n]$. We have therefore $\pi(i)\in S_i$ for all $i$ and all $\pi(i)$'s are distinct.
\end{proof}

\begin{lemma}\label{lem:sym_b}
The equality $r^j(\bb^S)=\bb^{\rho^j(S)}$ holds for every $j\in[n]$ and every nonempty subset $S$ of $[n]$.
\end{lemma}

\begin{proof}
We have 
$$r^j(\bb^S)=\frac 1 {|S|} \sum_{i\in S}r^j(\ee_i)=\frac 1 {|S|} \sum_{i\in S}\ee_{\rho^j(i)}=\frac 1 {|\rho^j(S)|} \sum_{i\in \rho^j(S)}\ee_i=\bb^{\rho^j(S)},$$ where the penultimate equality comes from the fact that $\rho^j$ is a permutation of $[n]$.
\end{proof}

%\begin{remark}
%Lemmas~\ref{lem:sperner_det} and~\ref{lem:cover} together provide an alternate proof of a generalization of the Sperner lemma due to Shapley~\cite{Shapley}.
%\end{remark}
%{\blue Shira: I think that the theorem of Shapley is that there exists a simplex in the triangulation with a balanced collection of vertex labelings, which is different from the condition that it is possible to pick a distinct index in each of the simplex's vertex labeling. For example, the set of labelings $\{3,3,12\}$ is balanced for $n=3$, but it is not possible to pick distinct labels. On the other hand, when $n=4$ we can pick distinct labels in the set $\{123,123,123,14\}$, but it is not balanced.}   

We end this section with a simple property of the map $r^j$.

\begin{lemma}\label{lem:detr}
We have $\det(r^j)=(-1)^{j-1}$ for all $j$.
\end{lemma}

\begin{proof}
The determinant is the sign of the permutation $\rho^j$. Denote by $\tau^j$ the transposition that interchanges $j$ and $j+1$. We have $\rho^{j+1}=\tau^j\circ\rho^j$. The conclusion follows from the fact that the sign of $\rho^1$ is $1$ (it is the identity).
\end{proof}

\subsection{Proofs}\label{proofs}

\begin{figure}
\begin{center}
\includegraphics[width=15cm]{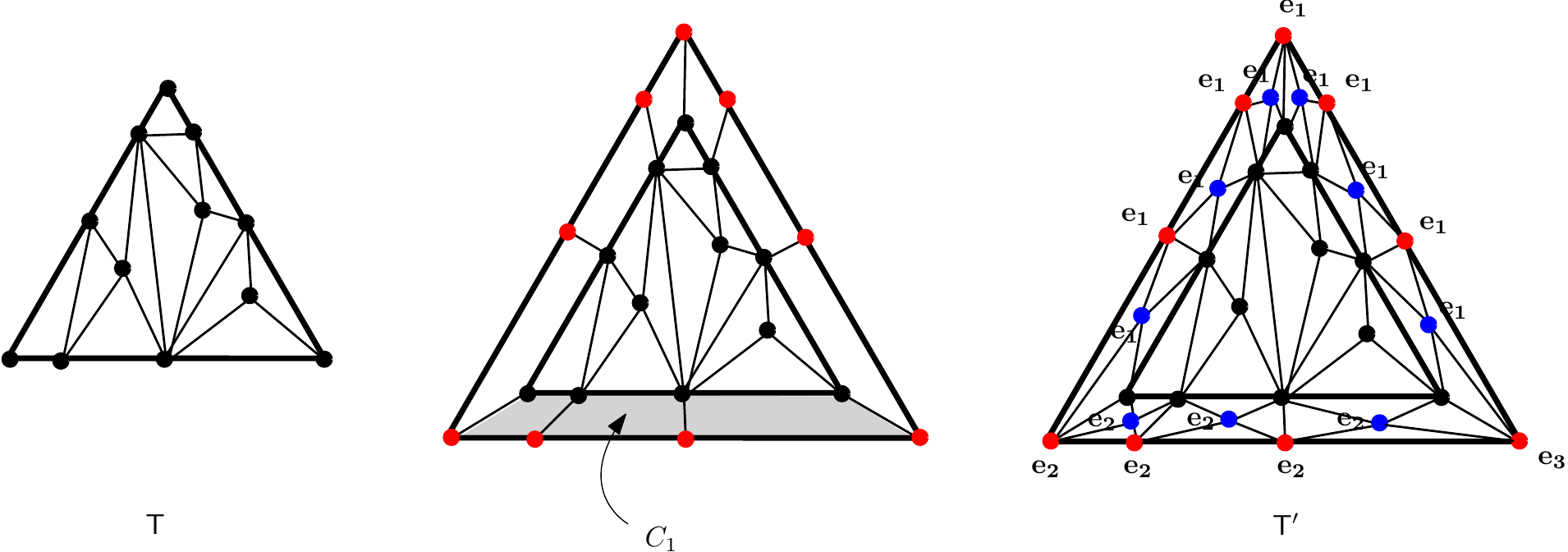}
\caption{Illustration of the construction of the triangulation $\T'$ from the triangulation $\T$ in the proof of Theorem~\ref{thm:sperner-symmetry-prime} and Proposition~\ref{prop:sperner-symmetry-4}. On the right-hand figure, the labels on the vertices in $V(\T')\setminus V(\T)$ have been displayed.}
\label{fig:triangulation}
\end{center}
\end{figure}

\begin{proof}[Proof of Theorem~\ref{thm:sperner-symmetry-prime}]
We put a shrunk copy of $\T$ inside $\Delta^{n-1}$. The non triangulated part of $\Delta^{n-1}$ admits a natural decomposition into $n$ cells, each being the convex hull $C_j$ of a $\widehat j$-facet and its shrunk copy. We complete $\T$ in a triangulation $\T'$ of $\Delta^{n-1}$ such that  $r^j(\sigma)\in\T'$ for every $j\in[n]$ and every simplex $\sigma\in\T'$ included in $C_1$. In particular, $\T'$ is nice. The triangulation of $C_j$ induced by $\T'$ is denoted $\C_j$.

(Such a triangulation $\T'$ is easily achieved by keeping on $\partial\Delta^{n-1}$ the triangulation induced by $\T$ before it was shrunk, and then proceeding to a subdivision of the prisms induced by the pairs $\sigma',\sigma$, where $\sigma'$ is an $(n-2)$-dimensional simplex of $\partial\Delta^{n-1}$ and $\sigma$ is its shrunk copy. This latter subdivision can be performed without adding new vertices in those simplices $\sigma,\sigma'$.)

Now, we label the vertices $v$ of $\T$ by $\lambda(v)=\bb^{\Lambda(v)}$. For the other vertices $u$ in $V(\T')\setminus V(\T)$, we proceed as follows: we consider the radial projection $u'$ from the barycenter of $\Delta^{n-1}$ on $\partial\Delta^{n-1}$. We label then $u$ by $\lambda(u)=\bb^{\{i(u)\}}=\ee_{i(u)}$, where $i(u)$ is the minimal index of a nonzero coordinate of the projection $u'$. (The construction of $\T'$ and of $\lambda$ is illustrated for $n=3$ on Figure~\ref{fig:triangulation}.)

The map $\lambda$ satisfies obviously the condition of Lemma~\ref{lem:sperner_det}. We have thus 
\begin{equation}\label{tprime} \left|(\det_{\sharp}\circ\lambda_{\sharp})(t')\right|=1,\end{equation} where $t'$ is the formal sum of all positively oriented $(n-1)$-dimensional simplices of $\T'$. (As in the proof of Lemma~\ref{lem:sperner_det}, the map $\lambda$ is seen as a simplicial map $\T'\rightarrow\Rsf$.)

On the other hand, we have by linearity 
\begin{equation}\label{decomp}
(\det_{\sharp}\circ\lambda_{\sharp})(t')=(\det_{\sharp}\circ\lambda_{\sharp})(t)+\sum_{j=1}^n(\det_{\sharp}\circ\lambda_{\sharp})(c_j),
\end{equation} where $t$ is the formal sum of all positively oriented $(n-1)$-dimensional simplices of $\T$ and $c_j$ the formal sum of all positively oriented $(n-1)$-dimensional simplices of $\C_j$.  As for other linear maps already met in this paper, we also  interpret $r^j$ as a simplicial self-map of $\Rsf$. Moreover, $r^j$ satisfies the following property.

\begin{claim}
The relation $(\lambda\circ r^j)(v)=(r^j\circ\lambda)(v)$ is satisfied for all $j$ and all $v\in V(\C_1)$.
\end{claim}
\begin{proof}
To ease the reading, we denote by $\Omega$ the barycenter of $\Delta^{n-1}$ (which is denoted $\bb^{[n]}$ elsewhere in the paper).

For $\alpha>0$, the application $a_{\alpha}$ that maps a point $\xx$ of $\Delta^{n-1}$ to a point $\xx'$ such that $\xx'=\Omega+\alpha(\xx-\Omega)$ commutes with $r^j$. This is because $r^j(\Omega)=\Omega$. When we shrink $\T$, we are applying such a map and Lemma~\ref{lem:sym_b} implies then that $(\lambda\circ r^j)(v)=(r^j\circ\lambda)(v)$ for all $v\in V(\C_1)\cap V(\T)$.

Consider now a vertex $v\in V(\C_1)$ that is on the boundary of $\T'$ and an integer $j\in[n]$. The $i$th and $i'$th coordinates of $v$ are nonzero if and only if the $\rho^j(i)$ and $\rho^j(i')$ coordinates of $r^j(v)$ are nonzero. In such a case, we have moreover that $i<i'$ if and only if $\rho^j(i)<\rho^j(i')$. It implies that $(\lambda\circ r^j)(v)=(r^j\circ\lambda)(v)$ for all $v\in V(\C_1)$ that are on the boundary of $\T'$.

Finally, since $a_{\alpha}$ and $r^j$ commutes, this holds for any other vertex in $V(\C_1)$ as well.
\end{proof}

We have 
$$\begin{array}{rcl}
\ds{\sum_{j=1}^n(\det_{\sharp}\circ\lambda_{\sharp})(c_j)} & = & \ds{\sum_{j=1}^n(-1)^{j-1}(\det_{\sharp}\circ\lambda_{\sharp}\circ r^j_{\sharp})(c_1)}\smallskip \\ 
& = & \ds{\sum_{j=1}^n(-1)^{j-1}(\det_{\sharp}\circ\, r^j_{\sharp}\circ\lambda_{\sharp})(c_1)}\smallskip \\ 
& = & \ds{\sum_{j=1}^n(\det_{\sharp}\circ\lambda_{\sharp})(c_1)}\smallskip\\ 
& = & n(\det_{\sharp}\circ\lambda_{\sharp})(c_1),
\end{array}$$
where the first equality is a consequence of the relation $r^j_{\sharp}(c_1)=(-1)^{j-1}c_j$, due to the definition of $r^j$ and to Lemma~\ref{lem:detr}, the second equality comes from the claim above, and the third one is again a consequence of Lemma~\ref{lem:detr}.

Since $\Lambda(v)$ is always a proper subset of $[n]$, the rational number $(\det\circ\lambda)(\sigma)$ can always be written as a fraction of integers, with a denominator being a product of integers smaller than $n$. With Equalities~\eqref{tprime} and~\eqref{decomp}, the fact that $n$ is a prime implies that $(\det_{\sharp}\circ\lambda_{\sharp})(t)$ is nonzero. There is therefore at least one $(n-1)$-dimensional simplex $\tau=[v_1,\ldots,v_n]$ of $\T$ such that $\det(\lambda(v_1),\ldots,\lambda(v_n))\neq 0$. Lemma~\ref{lem:cover} shows then we can pick distinct labels in the $\Lambda(v_i)$'s when $i$ runs over the integer $1,\ldots,n$.
\end{proof}

\begin{proof}[Proof of Proposition~\ref{prop:sperner-symmetry-4}]
The proof is exactly the same as the one of Theorem~\ref{thm:sperner-symmetry-prime}, except the last paragraph. Under the condition of Proposition~\ref{prop:sperner-symmetry-4}, the rational number $(\det\circ\lambda)(\sigma)$ can always be written as a fraction of integers, with a denominator equal to $3!=6$. With Equalities~\eqref{tprime} and~\eqref{decomp}, the fact that $4$ does not divide $6$ implies that $(\det_{\sharp}\circ\lambda_{\sharp})(t)$ is nonzero. The conclusion is then identical.
\end{proof}

%\begin{remark}
%In the proofs of both Theorem~\ref{thm:sperner-symmetry-prime} and Proposition~\ref{prop:sperner-symmetry-4}, the conclusion follows from the fact that $n$ does not divide $(n-1)!$, which happens precisely when $n$ is a prime number or is equal to $4$. This is why we have not been able to extend the approach for other values of $n$.
%\end{remark}

\begin{remark}
The proof above for Proposition~\ref{prop:sperner-symmetry-4} can be extended for values of $n$ different from $4$, but to conclude that $(\det_{\sharp}\circ\lambda_{\sharp})(t)$ is nonzero we need that $n$ does not divide $(n-1)!$, which happens precisely when $n$ is a prime number or is equal to $4$. Therefore, it does not seem that we can reach other values for $n$ with the current approach.
\end{remark}

\begin{remark}\label{secretive}
A ``secretive-player'' generalization of the classical envy-free cake division theorem has recently drawn some attention. It was originally proved by Woodall~\cite{woodall1980dividing} and rediscovered with a much simpler proof by Asada et al.~\cite{asada2017fair}, who also gave it its picturesque name. It states that an envy-free division can be achieved in the classical setting without taking into account the preferences of one fixed (``secretive'') player: there is a division such that no matter which piece is chosen by this player, there will be an envy-free assignment of the remaining pieces to the other players.

It is reasonable to assume that Theorem~\ref{main} also admits a ``secretive'' generalization. However, we do not know how to prove this using our approach. The natural adaptation of the proof by Asada et al. to our setting would require more from the simplex $\tau$ found in Theorem~\ref{thm:sperner-symmetry-prime}: denoting its vertices by $v_1,\ldots,v_n$, it would require that the barycenter of $\Delta^{n-1}$ is in the convex hull of the $\lambda(v_i)$ (where $\lambda$ is defined as in the proof of Theorem~\ref{thm:sperner-symmetry-prime}). In our proof we get that the determinant of the matrix whose columns are the $\lambda(v_i)$ is nonzero, but this does not imply this additional required property.
\end{remark}

\section{Proof of the main theorem}

%We give now the proof of Theorem~\ref{main}, which relies on Propositions~\ref{prop:sperner-symmetry-prime} and~\ref{prop:sperner-symmetry-4}. 

With Theorem~\ref{thm:sperner-symmetry-prime} and Proposition~\ref{prop:sperner-symmetry-4}, the proof of Theorem~\ref{main} is more or less routine. 
We start with a lemma that shows that triangulations satisfying the symmetry condition of Theorem~\ref{thm:sperner-symmetry-prime} and Proposition~\ref{prop:sperner-symmetry-4} exist 
%{\blue Shira: erase ``-- we called them ``nice'' triangulations in Section~\ref{sec:sperner} --"} 
and can have arbitrary small mesh size. This lemma shows moreover that we can label the vertices of the triangulation with the players in a way compatible with the symmetry. A labeling with the players  is called  ``owner-labeling'' by Su~\cite{su1999rental}, and ``ownership-assignment'' by Segal-Halevi~\cite{segal2017fairly}.

\begin{lemma}\label{lem:sym-own}
There exists a nice triangulation $\T$ of $\Delta^{n-1}$ of arbitrary small mesh size and a labeling $o\colon V(\T)\rightarrow[n]$ satisfying $o(r^j(v))=o(v)$ for every vertex $v$ of $\T$ in the $\widehat{1}$-facet and such that adjacent vertices in $\T$ get distinct labels via $o$.

Moreover, such a triangulation can be built so that the supporting faces of any two adjacent vertices are comparable by inclusion. 
\end{lemma}

\begin{proof}
We repeat the barycentric subdivision operation starting with $\Delta^{n-1}$ as many times as needed to get a triangulation $\T$ with a sufficiently small mesh size. This triangulation is clearly nice. 

The triangulation $\T$ is thus of the form $\sd^N\left(\Delta^{n-1}\right)$ for some positive integer $N$. Each vertex $v$ of $\T$ corresponds to a simplex of $\sd^{N-1}\left(\Delta^{n-1}\right)$. Defining $o(v)$ to be the dimension of this simplex plus one shows that this labeling $o$ is as required.
\end{proof}

\begin{proof}[Proof of Theorem~\ref{main}]
Consider a nice triangulation $\T$ with a labeling $o$ as in Lemma~\ref{lem:sym-own}. For a point $\xx$ in $\Delta^{n-1}$, we denote by $\D(\xx)$ the division of the cake obtained when the cake is cut at positions $X_1,\ldots,X_n$, where $X_{\ell}=\sum_{i=1}^{\ell}x_i$. (Remember that because of our metrics, whether the endpoints belong or not to the pieces does not matter.)

We are going to define a labeling $\Lambda$ of the vertices of $\T$ to which we are going to apply Theorem~\ref{thm:sperner-symmetry-prime} when $n$ is a prime number and Proposition~\ref{prop:sperner-symmetry-4} when $n=4$. This labeling will be defined with the help of a set-valued map $L_i\colon\Delta^{n-1}\rightarrow 2^{[n]}\setminus\{\varnothing\}$ we introduce now.

Consider a point $\xx$ in $\Delta^{n-1}$. For each nonempty piece $I$ in $p_{i}\left(\D(\xx)\right)$, we put in $L_i(\xx)$ the smallest index $j$ such that $X_j$ is the right endpoint of $I$. If $\varnothing\in p_{i}\left(\D(\xx)\right)$, then we add to $L_i(\xx)$ the set $J_{\xx}$. Because of the full division assumption, $L_i(\xx)\neq\varnothing$.

We define then 
$$\Lambda_i(\xx)=\left\{\begin{array}{ll} J_{\xx} & \mbox{if $L_i(\xx)=J_{\xx}$} \\ \{\min\left(L_i(\xx)\setminus J_{\xx}\right)\} & \mbox{otherwise.}\end{array}\right.$$
Note that if $L_i(\xx)\neq J_{\xx}$, then $L_i(\xx)\setminus J_{\xx}\neq\varnothing$, which ensures that $\Lambda_i(\xx)$ is always either $J_{\xx}$ or a singleton made of an element not in $J_{\xx}$.

%For each nonempty piece $I$ in $p_{i}\left(\D(\xx)\right)$, we put in $\Lambda_i(\xx)$ the smallest index $j$ such that $X_j$ is its right endpoint. Note that in this case the Lebesgue measure of $I$ is equal to $x_j$. If $\varnothing\in p_{i}\left(\D(\xx)\right)$, then we add moreover all $j$ such that $\xx$ belongs to the $\widehat{j}$-facet (because of the full division assumption, if $\varnothing\in p_{i}\left(\D(\xx)\right)$, there is at least one such $j$). In other words, if the point $\xx$ is on a face of $\Delta^{n-1}$ included in the $\widehat{j}$-facet, $j$ is in $\Lambda_i(\xx)$ if and only if player $i$ is happy with the empty piece in the division $\D(\xx)$.

\begin{claim}
Given $\xx$ in the $\widehat{1}$-facet of $\Delta^{n-1}$, the equality $\Lambda_i(r^j(\xx))=\rho^j\left(\Lambda_i(\xx)\right)$ holds for all $i$ and $j$.
\end{claim}

\begin{proof}
Let $\yy=r^j(\xx)$ and $Y_{\ell}=\sum_{i=1}^{\ell}y_i$. Note that $\D(\xx)=\D(\yy)$. 

Suppose first that there is at least one nonempty piece in $p_i(\D(\xx))$.
If $\ell$ is the smallest index such that $X_{\ell}$ is the right endpoint of a nonempty piece, then $\rho^j(\ell)$ is the smallest index $\ell'$ such that $Y_{\ell'}$ is the right endpoint of that same nonempty piece. This shows that in this case $\Lambda_i(\yy)=\rho^j\left(\Lambda_i(\xx)\right)$. 

Suppose now that $p_i(\D(\xx))=\{\varnothing\}$. We have $\Lambda_i(\xx)=J_{\xx}$ and $\Lambda_i(\yy)=J_{\yy}$. Since $J_{\yy}=\rho^j(J_{\xx})$ by definition of $\rho^j$, we have again $\Lambda_i(\yy)=\rho^j\left(\Lambda_i(\xx)\right)$.
\end{proof}

%\begin{proof}
%Let $\xx$ be a point in the $\widehat{1}$-facet of $\Delta^{n-1}$ and denote by $\yy$ its image by $r^j$. %By definition of $r^j$, we have the equivalence $x_{\ell}>0\Longleftrightarrow y_{\rho^j(\ell)}>0$. 
%
%Suppose first that $x_{\ell}>0$. In this case, $(X_{\ell-1},X_{\ell})$ is an nonempty interval and $(X_{\ell-1},X_{\ell})=(X_{\rho^j(\ell)-1},X_{\rho^j(\ell)})$. Thus, we have $\Lambda_i(\yy)=\{\rho^j(\ell)\}$ if and only if $\Lambda_i(\xx)=\{j\}$.
%
%Suppose then that $x_{\ell}=0$. It means that $\xx$ belongs to the $\widehat{\ell}$-facet and that $\yy$ belongs to the $\widehat{\rho^j(\ell)}$-facet. Thus, again, we have $\rho^j(\ell)\in\Lambda_i(\yy)$ if and only if $\ell\in\Lambda_i(\xx)$ (either player $i$ is happy with an empty piece or she is not).%Let $\ell\in\Lambda_i(\xx)$. 
%%Suppose first that $x_{\ell}>0$. In this case, $(X_{\ell-1},X_{\ell})$ is an interval player $i$ is happy with in the division $\D(\xx)$. Since $(X_{\ell-1},X_{\ell})=(X_{\rho^j(\ell)-1},X_{\rho^j(\ell)})$, we have $\rho^j(\ell)\in\Lambda_i(\yy)$.
%\end{proof}

For a vertex $v\in V(\T)$ of coordinate $\xx$, we define $\Lambda(v)$ to be $\Lambda_{o(v)}(\xx)$. %The triangulation $\T$ satisfies the condition of Theorem~\ref{thm:sperner-symmetry} because of \ref{a} in Lemma~\ref{lem:sym-own}. `
If $v$ is in the $\widehat{1}$-facet, we have
$$\Lambda(r^j(v))=\Lambda_{o(r^j(v))}(r^j(\xx))=\Lambda_{o(v)}(r^j(\xx))=\rho^j\left(\Lambda_{o(v)}(\xx)\right)=\rho^j\left(\Lambda(v)\right),$$
where $\xx$ is the coordinate vector of $v$. The first equality is by definition, the second is by Lemma~\ref{lem:sym-own}, the third is the preceding claim, and the last is again by definition.
 Thus, $\Lambda$ is a nice labeling. Moreover, it satisfies the additional condition of Proposition~\ref{prop:sperner-symmetry-4}. Theorem~\ref{thm:sperner-symmetry-prime} and Proposition~\ref{prop:sperner-symmetry-4} can be applied and their conclusion holds: there is an $(n-1)$-dimensional $\tau$ of $\T$ such that it is possible to pick a distinct label in each $\Lambda(u)$ when $u$ runs over the vertices of $\tau$.

Lemma~\ref{lem:sym-own} allows to choose $\T$ of arbitrarily small mesh size. Compactness and the following claim imply that there is a point $\xx^*=(x_1^*,\ldots,x_n^*)$ of $\Delta^{n-1}$ such that it is possible to select a distinct element from each $L_i(\xx^*)$ when $i$ goes from $1$ to $n$. 

\begin{claim}
Let $(\xx^k)$ be a sequence of points in $\Delta^{n-1}$ converging to a point $\xx^{\infty}$. If $j\in\Lambda_i(\xx^k)$ for all $k$, then $j\in L_i(\xx^{\infty})$.
\end{claim}

\begin{proof}
%Let $J^k$ (resp. $J_{\xx_{\infty}}$) be the set of integers $j$ such that $\xx^k$ (resp. $\xx^{\infty}$) belongs to the $\widehat{j}$-facet. 

We assume without loss of generality that all $J_{\xx^k}$ are equal (for finite $k$) and denote by $J$ this subset. We have $J_{\xx^{\infty}}\supseteq J$. Consider a $j$ that is in all $\Lambda_i(\xx^k)$.

Suppose first that $j\notin J_{\xx^{\infty}}$. The intervals $(X_{j-1}^k,X_j^k)$ and $(X_{j-1}^{\infty},X_j^{\infty})$ are all of positive length. The interval $(X_{j-1}^k,X_j^k)$ belongs thus to $p_i(\D(\xx^k))$ for all $k$ and since $$\lim_{k\rightarrow+\infty}\delta\left((X_{j-1}^k,X_j^k),(X_{j-1}^{\infty},X_j^{\infty})\right)=0,$$ the interval $(X_{j-1}^{\infty},X_j^{\infty})$ belongs to $p_i(\D(\xx^{\infty}))$ (closed preferences assumption). The interval $(X_{j-1}^{\infty},X_j^{\infty})$ being of positive length, we get $j\in L_i(\xx^{\infty})$.

Suppose now that $j\in J_{\xx^{\infty}}\setminus J$. The intervals $(X_{j-1}^k,X_j^k)$ are all of positive length. The interval $(X_{j-1}^k,X_j^k)$ belongs thus to $p_i(\D(\xx^k))$ for all $k$. The fact that $j\in J_{\xx^{\infty}}$ means that $X_{j-1}^{\infty}=X_j^{\infty}$, which implies that 
$$\lim_{k\rightarrow+\infty}\delta\left((X_{j-1}^k,X_j^k),\varnothing\right)=0.$$ The closed preferences assumption implies then that $\varnothing\in p_i(\D(\xx^{\infty}))$. In such a case, by definition of $L_i$, we have $j\in L_i(\xx^{\infty})$.

Suppose finally that $j\in J$. By definition of $\Lambda_i$, it means that $\varnothing\in p_i(\D(\xx^k))$ for all $k$. Since $\delta(\varnothing,\varnothing)=0$, we get that $\varnothing\in p_i(\D(\xx^{\infty})$ and thus $j\in L_i(\xx^{\infty})$.
\end{proof}

We finish the proof by showing that $\D(\xx^*)$ is an envy-free division.  Denote by $j_i$ pairwise distinct elements selected in the $L_i(\xx^*)$ when $i$ goes from $1$ to $n$. If $(X_{j_i-1}^*,X_{j_i}^*)\in p_i(\D(\xx^*))$, define $\pi(i)=(X_{j_i-1}^*,X_{j_i}^*)$. Otherwise, define $\pi(i)=\varnothing$. By definition of $L_i$, we have $\pi(i)\in p_i(\D(\xx^*))$ (item~\ref{i} is satisfied). Since $\{j_i\colon i\in[n]\}=[n]$, each nonempty piece is equal to some $\pi(i)$ (item~\ref{ii} is satisfied). Finally, if $\pi(i)=\pi(i')$, we have $\pi(i)=\varnothing$ because otherwise $j_i$ would be equal to $j_{i'}$ (item~\ref{iii} is satisfied). 
\end{proof}

\subsection*{Acknowledgments} 
The authors are grateful to the referee for his thorough reading and his suggestions and questions that helped improve the paper. 

This work has been initiated when the authors were in residence at the Mathematical Sciences Research Institute in Berkeley, California, during the Fall 2017 semester. This material is  thus partially based upon work supported by the National Science Foundation under Grant No. DMS-1440140.

\bibliographystyle{plain}
\bibliography{cake}

\end{document}